\documentclass[reqno,A4paper]{amsart}
\usepackage{amsmath}
\usepackage{amssymb}
\usepackage{amsthm}
\usepackage{enumerate}
\usepackage{amsmath}
\usepackage{mathrsfs} 
 \usepackage{color}                            
\usepackage{eqlist}
\usepackage{array}

\theoremstyle{plain}
\newtheorem{theorem}{Theorem}[section]

\newtheorem{lemma}[theorem]{Lemma}
\newtheorem{corollary}[theorem]{Corollary}

\theoremstyle{definition}

\newtheorem{remark}{\textup{Remark}} 

\numberwithin{equation}{section}

\calclayout

\title[Characterizations of (Jordan) derivation on Banach algebra with local actions]{Characterizations of (Jordan) derivation on Banach algebra with local actions}

\author[]{Jiankui Li}
\author[]{Shan Li}
\author[]{Kaijia Luo}

\address{Jiankui Li, East China University  of Science and Technology, Shanghai, 200237, China}
\email{jkli@ecust.edu.cn}
\address{Shan Li, Jiangsu University of Technology, Changzhou, 213001, China}
\email{lishan\_math@163.com}
\address{Kaijia Luo, East China University  of Science and Technology, Shanghai, 200237, China}
\email{kaijia\_luo@163.com}

\keywords{Banach algebra, Derivation, Jordan derivation, Separating point}

\subjclass{47B47, 15A86}
\begin{document}

\begin{abstract}
Let $\mathcal{A}$ be a unital Banach $*$-algebra and $\mathcal{M}$ be a unital $*$-$\mathcal{A}$-bimodule. If $W$ is a left  separating point of $\mathcal{M}$, we show that every $*$-derivable mapping at $W$ is a Jordan derivation, and every $*$-left derivable mapping at $W$ is a Jordan left derivation under the condition $W \mathcal{A}=\mathcal{A}W$. Moreover we give a complete description of linear mappings $\delta$ and $\tau$ from $\mathcal{A}$ into $\mathcal{M}$ satisfying $\delta(A)B^*+A\tau(B)^*=0$ for any $A, B\in \mathcal{A}$ with $AB^*=0$ or
$\delta(A)\circ B^*+A\circ\tau(B)^*=0$  for any $A, B\in \mathcal{A}$ with $A\circ B^*=0$, where $A\circ B=AB+BA$ is the Jordan product.

\end{abstract}

\maketitle

\section{Introduction}
Let $\mathcal{A}$ be an algebra over a real or complex field $K$ and $\mathcal{M}$ be an $\mathcal{A}$-bimodule. A linear mapping $\delta$ from $\mathcal{A}$ into $\mathcal{M}$ is called a \emph{derivable mapping at $W$ }if $\delta(W)=A\delta(B)+\delta(A)B$ for all $A ,B\in \mathcal{A}$ with $AB=W$.
As is well known, the problem of linear mappings preserving fixed products is a very interesting item in the field of operator algebra.
Derivations that can be completely determined by the local action on some subsets of algebra  have attracted attention of many researchers.
Here, we only focus on derivable mappings at special points.
There are a considerable number of influential results on derivable mappings at zero, unit, invertible elements, separating points, idempotent elements and so on.
In some results, a derivable mapping at zero is described in terms of a generalized derivation.
Jing, Lu and Li \cite{JLL} described  bounded derivable mappings at zero on von Neumann algebras.
Bre\v{s}ar \cite{B} studied additive derivable mappings at zero on rings with nontrivial idempotents.
In \cite{L}, Lu considered continuous derivable mappings from a unital Banach algebra into its unital Banach bimodule at left or right separating points.
Furthermore, the first author of the paper and Zhou \cite{LZ} generalized this result without the assumption of continuity.

There is also a related variation of derivable mappings  on  $*$-algebras. A linear mapping $\delta$ from a $*$-algebra $\mathcal{A}$ into a $*$-$\mathcal{A}$-bimodule $\mathcal{M}$ is a \emph{$*$-derivable mapping at $W$}, if
\begin{align*}
A,B\in \mathcal{A},\ AB^*=W &\quad \Longrightarrow \quad \delta(W)=A\delta(B)^*+\delta(A)B^*.\tag{$\mathbb{D}_W^*$}
\end{align*}
$*$-derivable mappings at zero  were studied on unital $*$-algebras, von Neumann algebras, $C^*$-algebras, and zero product determined $*$-algebras in \cite{AHL2,FFG,FG}.
Fadaee, Fallahi and Ghahramani \cite{FFG} considered  continuous $*$-derivable mappings at left separating points from a unital Banach $*$-algebra into its Banach $*$-bimodule.
In \cite{AHL}, An, He and the first author of this paper proved that every (continuous) $*$-derivable mapping  from a (unital $C^*$-algebra) factor von Neumann algebra into its Banach $*$-bimodule is a $*$-derivation if and only if $W$ is a left separating point. Recall that a derivation $\delta$ from $\mathcal{A}$ into $\mathcal{M}$ is a \emph{$*$-derivation} if $\delta(A^*) = \delta(A)^*$ for each $A$ in $\mathcal{A}$.
Our first goal is to give a description of linear (unnecessary continuous assumption) mappings satisfying $\mathbb{D}_W^*$ from a Banach $*$-algebra into its $*$-bimodule.

Some natural development of usual $*$-derivable mappings is $*$-left derivable mappings. A linear mapping $\delta$ from a $*$-algebra $\mathcal{A}$ into a $*$-$\mathcal{A}$-bimodule $\mathcal{M}$  is a \emph{$*$-left derivable mapping at $W$}, if
\begin{align*}
A,B\in \mathcal{A},\ AB^*=W &\quad \Longrightarrow  \quad \delta(W)=A\delta(B)^*+B^*\delta(A).\tag{$\mathbb{L}_W^*$}
\end{align*}
In \cite{AHL}, the authors showed that every continuous $*$-left derivable mapping at  separating points from a unital $C^*$-algebra or a von Neumann algebra into its unital Banach $*$-bimodule is identical with zero under certain conditions.
Our second aim is to give the representation of linear mappings satisfying the property $\mathbb{L}^*_W$.

Based on above facts, there is a related problem: how to characterize two linear mappings $\delta$ and $\tau$  from $\mathcal{A}$ into $\mathcal{M}$ satisfying $\delta(A)B+A\tau(B)=0$ for each $A, B \in \mathcal{A}$ with $AB=0$ ? Benkovi\v{c} and Gra\v{s}i\v{c} \cite{BG} studied this problem on unital zero product determined algebras (see Section 4).
The first two authors of this paper \cite{LL} considered it from $\mathcal{A}$ into $\mathcal{M}$ with the property $\mathbb{P}$ (see Section 4).
 Next, we consider linear mappings $\delta$ and $\tau$ from a $*$-algebra $\mathcal{A}$ into a $*$-$\mathcal{A}$-bimodule $\mathcal{M}$ satisfying
\begin{align*}
A,B\in \mathcal{A},\ AB^*=0 &\quad \Longrightarrow  \quad 0=\delta(A)B^*+A\tau(B)^*.\tag{$\mathbb{T}_0^*$}
\end{align*}
To our knowledge, a few work has been done up to now. Our next goal is to characterize linear mappings satisfying $\mathbb{T}_0^*$.

Meanwhile, some papers \cite{ABEV,G2,Liu} gave the structure of derivable mappings at zero Jordan products. Jordan product is denoted by `` $\circ$ " : $A\circ B=AB + BA$. A linear mapping $\delta$ from $\mathcal{A}$ into $\mathcal{M}$ is a \emph{Jordan derivable mapping at zero Jordan products} if $\delta(A)\circ B+A\circ \delta(B)=0$ for each $A, B$ in $\mathcal{A}$ with $A\circ B=0$.
 Our last target is to discuss two linear mappings $\delta$ and $\tau$ from a $*$-algebra $\mathcal{A}$ into a $*$-$\mathcal{A}$-bimodule $\mathcal{M}$ satisfying
 \begin{align*}
 A,B\in \mathcal{A},\ A\circ B^*=0& \quad \Longrightarrow \quad  0=\delta(A)\circ B^*+A\circ\tau(B)^*.\tag{$\mathbb{J}_0^*$}
\end{align*}

Let us fix some more notations.
A linear mapping $\delta$ from $\mathcal{A}$ into $\mathcal{M}$ is called a \emph{derivation} if $\delta(AB)=\delta(A)B+A\delta(B)$ for each $A,B \in \mathcal{A}$; $\delta$ is called a \emph{Jordan derivation} if $\delta(A^2)=\delta(A)A+A\delta(A)$  for each $A \in \mathcal{A}$, which is equivalent to $\delta(A\circ B)=\delta(A)\circ B+A\circ \delta(B)$ for each $A, B$ in $\mathcal{A}$.
It is clear that every derivation is a Jordan derivation. But the reverse is not always true.
Besides, a linear mapping $\delta$ from $\mathcal{A}$ into $\mathcal{M}$ is called a \emph{Jordan left derivation} if $\delta(A^2)=2A\delta(A)$ for each $A \in \mathcal{A}$.
An element $W$ in $\mathcal{A}$ is a \emph{left (right) separating point} of $\mathcal{M}$ if $WM=0$ (or $MW=0$) implies $M=0$ for each $M\in \mathcal{M}$. $W$ is called a \emph{separating point} if $W$ is both a left separating point and a right separating point. It is easy to see that left (right) invertible elements in $\mathcal{A}$ are left (right) separating points of $\mathcal{M}$, and invertible elements in $\mathcal{A}$ are separating points of $\mathcal{M}$.

\section{Property $\mathbb{D}_W^*$}\label{sec2}
In this section, we discuss a linear mapping from a unital Banach $*$-algebra $\mathcal{A}$ into a unital $*$-$\mathcal{A}$-bimodule $\mathcal{M}$ satisfying $\mathbb{D}_W^*$ without continuity.
The following lemma is very important in the proofs of our results.

\begin{lemma}\label{lem1}\cite[Lemma 2.1]{LZ}
If $\delta$ is a linear mapping from a unital Banach algebra $\mathcal{A}$ into a unital $\mathcal{A}$-bimodule $\mathcal{M}$ such that $A\delta(A^{-1})+\delta(A)A^{-1}=0$ for any invertible element $A$ in $\mathcal{A}$, then $\delta$ is a Jordan derivation.
\end{lemma}
\begin{theorem}\label{thD1*}
Suppose that $\mathcal{A}$ is a unital Banach $*$-algebra, $\mathcal{M}$ is a unital $*$-$\mathcal{A}$-bimodule, and $W$ is a left separating point of $\mathcal{M}$.
If $\delta$ is a linear mapping from $\mathcal{A}$ into $\mathcal{M}$ satisfying  $\mathbb{D}_W^*$, then $\delta$ is a Jordan derivation and satisfies $\delta(WA)=\delta(W)A+ W\delta(A^*)^*$ for each $A$ in $\mathcal{A}$.
\end{theorem}
\begin{proof}
From $IW=WI=W$, it follows that
$\delta(W)=\delta(I)W+{\delta(W^*)}^*$ and $\delta(W)=\delta(W)+W{\delta(I^*)}^*$.
Since $W$ is a left separating point of $\mathcal{M}$, we have  $$\delta(I)=0\quad \text{and}\quad\delta(W^*)={\delta(W)}^*.$$
For each invertible element $T$ in $\mathcal{A}$, $\delta(W)=\delta(WT^{-1}T)=\delta(WT^{-1})T+WT^{-1}{\delta(T^*)}^*$, it follows that
\begin{align}
\delta(WT^{-1})=\delta(W)T^{-1}-WT^{-1}{\delta(T^*)}^*T^{-1},\label{2.1}\\
WT^{-1}{\delta(T^*)}^*=\delta(W)-\delta(WT^{-1})T.\label{2.2}
\end{align}
Let $A\in\mathcal{A}$, $n \in \mathbb{N}$ with $n\geq \|A\|+2$, and $B=nI+A$. Then $B$ and $I-B$ are both invertible in $\mathcal{A}$. It follows from (\ref{2.1}) and (\ref{2.2}) that
\begin{align*}
WB^{-1}{\delta(B^*)}^*&=\delta(W)-\delta(WB^{-1})B\\
&=\delta(W)-\delta(WB^{-1}(I-B)+W)B\\
&=\delta(W)(I-B)-\delta(WB^{-1}(I-B))B\\
&=\delta(W)(I-B)-[\delta(W)B^{-1}(I-B)\\
&\quad -WB^{-1}(I-B){(\delta((I-B)^{-1}B)^*)}^*B^{-1}(I-B)]B\\
&=WB^{-1}(I-B){(\delta((I-B)^{-1}B)^*)}^*(I-B)\\
&=WB^{-1}(I-B){(\delta{(I-B)^{-1}}^*)}^*(I-B).
\end{align*}
Since $W$ is a left separating point of $\mathcal{M}$, one can obtain
$$B^{-1}{\delta(B^*)}^*=B^{-1}(I-B){(\delta{(I-B)^{-1}}^*)}^*(I-B).$$
Thus
\begin{align*}
{\delta(B^*)}^*=(I-B){(\delta{(I-B)^{-1}}^*)}^*(I-B).
\end{align*}
Multiplying $W$ from the left of the above equation,
\begin{align*}
W{\delta(B^*)}^*&=W(I-B){(\delta{(I-B)^{-1}}^*)}^*(I-B)\\
&=(\delta(W)-\delta(W(I-B))(I-B)^{-1})(I-B)\\
&=\delta(W)(I-B)-\delta(W(I-B))\\
&=-\delta(W)B+\delta(WB).
\end{align*}
Hence
\begin{align*}
\delta(WB)=\delta(W)B+W{\delta(B^*)}^*.
\end{align*}
It follows that
\begin{align*}
 \delta(WA)&=\delta(W)A+W{\delta(A^*)}^*
\end{align*}
for each $A$ in $\mathcal{A}$.

 For each invertible element $T$ in $\mathcal{A}$, by assumption,
\begin{align*}
\delta(W)&=\delta(WTT^{-1})=\delta(WT)T^{-1}+WT{\delta({T^{-1}}^*)}^*\\
&=(\delta(W)T+W\delta(T^*)^*)T^{-1}+WT{\delta({T^{-1}}^*)}^*\\
&=\delta(W)+W{{\delta(T^*)}^*}T^{-1}+WT{\delta({T^{-1}}^*)}^*.
\end{align*}
It means that $W{{\delta(T^*)}^*}{T^{-1}}+WT{\delta({T^{-1}}^*)}^*=0$. By assumption, ${{\delta(T^*)}^*}T^{-1}+T{\delta({T^{-1}}^*)}^*=0$. Then
\begin{align*}
\delta(T)T^{-1}+T\delta(T^{-1})=0.
\end{align*}
By Lemma \ref{lem1}, $\delta$ is a Jordan derivation.
\end{proof}

By Theorem \ref{thD1*}, we can show that the assumption of continuity of result in \cite[Theorem 2.3]{FFG} can be removed.
The following conclusion follows immediately from what we have proved.
\begin{corollary}
Suppose that $\mathcal{A}$ is a unital Banach $*$-algebra and $\mathcal{M}$ is a unital $*$-$\mathcal{A}$-bimodule.
If $\delta$ is a linear mapping from $\mathcal{A}$ into $\mathcal{M}$ satisfying  $\delta(I)=A\delta(B)^*+\delta(A)B^*$ for all $A,B$ in $\mathcal{A}$ with $AB^*=I$, then $\delta$ is a Jordan derivation.
\end{corollary}

If we impose a condition that every Jordan derivation from $\mathcal{A}$ into $\mathcal{M}$ is a derivation, we can get a stronger result.

\begin{theorem}\label{thD2*}
Let $\mathcal{A}$ be a unital Banach $*$-algebra and let $\mathcal{M}$ be a unital $*$-$\mathcal{A}$-bimodule.
Suppose that every Jordan derivation from $\mathcal{A}$ into $\mathcal{M}$ is a derivation.
If $\delta$ is a linear mapping from $\mathcal{A}$ into $\mathcal{M}$ satisfying $\mathbb{D}_W^*$, then $\delta$ is a $*$-derivation if and only if $W$ is a left separating point of $\mathcal{M}$.
\end{theorem}
\begin{proof}
First we prove sufficiency for the theorem.

Let $\delta$ be a linear mapping from $\mathcal{A}$ into $\mathcal{M}$ satisfying $\mathbb{D}_W^*$ and $W$ be a left separating point of $\mathcal{M}$. From Theorem \ref{thD1*}, $\delta$ is a Jordan derivation and satisfies $\delta(WA)=\delta(W)A+ W\delta(A^*)^*$ for each $A$ in $\mathcal{A}$.
By assumption, $\delta$ is a derivation.

Next, we prove $\delta(A^*)=\delta(A)^*$ for every $A$ in $\mathcal{A}$. Let $A\in\mathcal{A}$, $n \in \mathbb{N}$ with $n\geq \|A\|+2$, and $T=nI+A$. Then $T$ is invertible in $\mathcal{A}$.
By $WTT^{-1}=W$, we have
\begin{align*}
\delta(W)=\delta(WT)T^{-1}+WT{\delta({T^{-1}}^*)}^*=\delta(W)+ W\delta(T)T^{-1}+WT{\delta({T^{-1}}^*)}^*.
\end{align*}
Hence $W[\delta(T)T^{-1}+T{\delta({T^{-1}}^*)}^*]=0$, i.e.,$$\delta(T)T^{-1}+T{\delta({T^{-1}}^*)}^*=0.$$
Since $\delta$ is a derivation, we have
$$\delta(I)=\delta(T)T^{-1}+T\delta(T^{-1}).$$
Comparing above two equalities with $\delta(I)=0$,
\begin{align*}
T{\delta({T^{-1}}^*)}^*=T\delta(T^{-1}),
\end{align*}
i.e., $\delta(T^*)={\delta(T)}^*$.
Since $T=nI+A$, $\delta(A^*)={\delta(A)}^*$ for each $A$ in $\mathcal{A}$.
Thus, $\delta$ is a $*$-derivation.

Finally, we consider the necessity. If $W$ is not a left separating point of $\mathcal{M}$, there exists a non-zero element $M$ in $\mathcal{M}$ such that $WM=0$.

Define a linear mapping $\delta$ from $\mathcal{A}$ to $\mathcal{M}$ as
$$\delta(A)=M^*A$$
for each $A$ in $\mathcal{A}$. Then $\delta$ is a linear mapping and satisfies
$$A,B\in\mathcal{A},AB^*=W \Rightarrow A{\delta(B)}^*+\delta(A)B^*=\delta(W).$$
However, $\delta$ is not a derivation, which leads to a contradiction.
\end{proof}

By Theorem \ref{thD2*}, we know that if every Jordan derivation from $\mathcal{A}$ into $\mathcal{M}$ is a derivation, then every linear mapping from $\mathcal{A}$ into $\mathcal{M}$ satisfying $\mathbb{D}_W^*$ in which $W$ is left separating point of $\mathcal{M}$ is a $*$-derivation.
The problem when a Jordan derivation is a derivation has caught the attention of mathematicians.
Bre\v{s}ar \cite{B1} showed that every Jordan derivation on $2$-torsion free semiprime rings is a derivation.
Johnson showed that every continuous Jordan derivation from a $C$*-algebra into its Banach bimodule is a derivation in \cite[Theorem 6.3]{JS}.
Peralta and Russo \cite[Corollary 17]{PR} proved that every Jordan derivation from a $C$*-algebra into its Banach bimodule is continuous.
Then Theorem \ref{thD2*} generalizes the conclusion in \cite[Theorem 1]{AHL}.
Moreover, Ghahramani \cite{G3} asserted that any additive Jordan derivation from a full matrix ring over a unital ring into its $2$-torsion free bimodule (not necessarily unital) is an additive derivation.

Finally, we consider to substitute right separating points for left separating points in Theorem \ref{thD1*}.

\begin{corollary}
Suppose that $\mathcal{A}$ is a unital Banach $*$-algebra, $\mathcal{M}$ is a unital $*$-$\mathcal{A}$-bimodule, and $W$ is a right separating point of $\mathcal{M}$.
If $\delta$ is a linear mapping from $\mathcal{A}$ into $\mathcal{M}$ satisfying  $\mathbb{D}_W^*$ and $\delta(I)=0$, then $\delta$ is a Jordan derivation.
\end{corollary}
\begin{proof}
From $IW=W$, we have $\delta(W)=\delta(I)W+{\delta(W^*)}^*$.
Since $\delta(I)=0$, it follows that $\delta(W^*)={\delta(W)}^*$.
By assumption,
$$AB^*=W^* \Longrightarrow BA^*=W \Longrightarrow \delta(W)=B\delta(A)^*+\delta(B)A^*$$
$$ \Longrightarrow \delta(W)^*=\delta(A)B^*+A\delta(B)^*$$
Since $\delta(W^*)={\delta(W)}^*$,
$$AB^*=W^* \Longrightarrow \delta(W^*)=\delta(A)B^*+A\delta(B)^*$$
By Theorem \ref{thD1*}, $\delta$ is a Jordan derivation.
\end{proof}

\begin{corollary}
Let $\mathcal{A}$ be a unital Banach $*$-algebra and $\mathcal{M}$ be a unital $*$-$\mathcal{A}$-bimodule.
Suppose that every Jordan derivation from $\mathcal{A}$ into $\mathcal{M}$ is a derivation.
If $\delta$ is a linear mapping from $\mathcal{A}$ into $\mathcal{M}$ satisfying $\mathbb{D}_W^*$ and $\delta(I)=0$, then $\delta$ is a $*$-derivation if and only if $W$ is a right separating point of $\mathcal{M}$.
\end{corollary}
\begin{proof}
By Theorem \ref{thD2*}, $\delta$ is a $*$-derivation if and only if $W$ is a left separating point of $\mathcal{M}$.
By assumption, we have $AB^*=W^* \Longrightarrow \delta(W^*)=\delta(A)B^*+A\delta(B)^*$.
Let $G=W^*$, then the linear mapping $\delta$ satisfies $\mathbb{D}_G^*$.
Hence, $\delta$ is a $*$-derivation if and only if $G$ is a left separating point of $\mathcal{M}$.
It follows that $\delta$ is a $*$-derivation if and only if $W$ is a right separating point of $\mathcal{M}$.
\end{proof}

\begin{remark}\rm
Suppose that $\mathcal{A}$ is a unital Banach $*$-algebra and $\mathcal{M}$ is a unital $*$-$\mathcal{A}$-bimodule, then the following conditions are not equivalent:
\begin{align*}
AB^*=W &\Longrightarrow \delta(W)=\delta(A)B^*+A\delta(B)^*,\tag{D1}\label{D1} \\
A^*B=W &\Longrightarrow \delta(W)=\delta(A)^*B+A^*\delta(B).\tag{D2}\label{D2}
\end{align*}

If a linear mapping $\delta$ satisfies condition (\ref{D2}) and $W$ is a right separating point of $\mathcal{M}$, then by suitable modification to the proofs in Theorem \ref{thD1*}, we can obtain that $\delta$ is a Jordan derivation.

Suppose that every Jordan derivation from $\mathcal{A}$ into $\mathcal{M}$ is a derivation.
If $\delta$ is a linear mapping from $\mathcal{A}$ into $\mathcal{M}$ satisfying (\ref{D2}), then $\delta$ is a $*$-derivation if and only if $W$ is a right separating point of $\mathcal{M}$.
\end{remark}

\section{Property $\mathbb{L}_W^*$}\label{sec3}
In this section, we deal with the linear mappings satisfying $\mathbb{L}_W^*$ at left separating points and show that the continuity of the results in \cite[Theorem 3]{AHL} is not necessary.

Let us start with the lemma showed in \cite{DL}.

\begin{lemma}\label{lem2}\cite[Lemma 2.1]{DL}
Let $\mathcal{A}$ be a unital Banach algebra, $\mathcal{M}$ be a unital left $\mathcal{A}$-module, and $\delta: \mathcal{A}\rightarrow \mathcal{M}$ be a linear mapping. If for each invertible element
$A$ in $\mathcal{A}$, we have $A\delta(A^{-1})+A^{-1}\delta(A)=\delta(I)$, then $\delta$ is a Jordan left derivation.
\end{lemma}

The proof of the following result is similar to the proof of Theorem \ref{thD1*}. Nevertheless, we include a proof for completeness.

\begin{theorem}\label{L1*} Let $\mathcal{A}$ be a unital Banach $*$-algebra, $\mathcal{M}$ be a unital $*$-$\mathcal{A}$-bimodule, and $W$ be a fixed point in $\mathcal{A}$. If $\delta$ is a linear mapping from $\mathcal{A}$ into $\mathcal{M}$ satisfying $\mathbb{L}_W^*$ and $\delta(I)=0$, then $W\delta$ is a Jordan left derivation.
\end{theorem}
\begin{proof}
It follows from $WI=W$ that $\delta(W)=W\delta(I)+\delta(W^*)^*$.
Then
$$\delta(W^*)^*=\delta(W).$$
For each invertible element $T$ in $\mathcal{A}$, $\delta(W)=\delta(TT^{-1}W)=T^{-1}W\delta(T)+T{(\delta({T^{-1}}W)^*)}^*$, it follows that
\begin{align}
{(\delta{({T^{-1}}W)^*})}^*={T^{-1}}\delta(W)-{T^{-2}}W\delta(T), \label{3.1}\\
{T^{-1}}W\delta(T)=\delta(W)-T{(\delta({T^{-1}W)^*})}^*.\label{3.2}
\end{align}
Let $A\in\mathcal{A}$, $n \in \mathbb{N}$ with $n\geq \|A\|+2$, and $B=nI+A$. Then $B$ and $I-B$ are both invertible in $\mathcal{A}$. By (\ref{3.1}) and (\ref{3.2}),
\begin{align*}
{B^{-1}}W\delta(B)&=\delta(W)-B{(\delta(B^{-1}W)^*)}^*\\
&=(I-B)\delta(W)-B[{B^{-1}}(I-B)\delta(W)\\
&\quad -({B^{-1}}(I-B))^{2}W\delta((I-B)^{-1}B)]\\
&=(I-B){B^{-1}}(I-B)W\delta((I-B)^{-1}B)\\
&=(I-B){B^{-1}}[\delta(W)-(I-B)^{-1}{(\delta((I-B)W)^*)}^*]\\
&=-\delta(W)+{B^{-1}}{(\delta(BW)^*)}^*.
\end{align*}
Then we have
\begin{align*}
{(\delta(AW)^*)}^*=W{\delta(A)}+A\delta(W)
\end{align*}
for each $A$ in $\mathcal{A}$.

For each invertible element $T$ in $\mathcal{A}$,
\begin{align*}
\delta(W)&=\delta(TT^{-1}W)=T^{-1}W\delta(T)+T{(\delta(T^{-1}W)^*)}^*\\
&=T^{-1}W\delta(T)+T(W\delta(T^{-1})+T^{-1}\delta(W))\\
&=T^{-1}W\delta(T)+TW\delta(T^{-1})+\delta(W).
\end{align*}
That is, $T^{-1}W\delta(T)+TW\delta(T^{-1})=0=\delta(I)$ for each invertible element $T$ in $\mathcal{A}$. By Lemma \ref{lem2}, $W\delta$ is a Jordan left derivation.
\end{proof}

The following corollary is now evident from what we have proved.
\begin{corollary} Let $\mathcal{A}$ be a unital Banach $*$-algebra and $\mathcal{M}$ be a unital $*$-$\mathcal{A}$-bimodule. If $\delta$ is a linear mapping from $\mathcal{A}$ into $\mathcal{M}$ satisfying  $\delta(I)=A\delta(B)^*+B^*\delta(A)$ for all $A,B$ in $\mathcal{A}$ with $AB^*=I$, then $\delta$ is a Jordan left derivation.
\end{corollary}

Let us mention two important consequences of Theorem \ref{L1*}. We consider the condition that $W$ is a left separating point.
In \cite{ADL}, authors proved that every Jordan left derivation from $C^*$-algebra into its Banach left module is zero.
The following corollary extends the result in \cite[Theorem 3]{AHL}.

\begin{corollary} Let $\mathcal{A}$ be a unital $C^*$-algebra, $\mathcal{M}$ be a unital Banach $*$-$\mathcal{A}$-bimodule, and $W$ be a left separating point of $\mathcal{M}$. If $\delta$ is a linear mapping from $\mathcal{A}$ into $\mathcal{M}$ satisfying $\mathbb{L}_W^*$, then $\delta$ is identical with zero.
\end{corollary}
\begin{proof}By assumption, $W\delta$ is a Jordan left derivation.
Hence, $W\delta$ is identical with zero. i.e., $W\delta(A)=0$ for each $A$ in $\mathcal{A}$. Since $W$ is a left separating point of $\mathcal{M}$, $$\delta(A)=0$$ for each $A$ in $\mathcal{A}$. So $\delta$ is identical with zero.
\end{proof}

In \cite{V}, Vukman showed that every Jordan left derivation on a complex semisimple Banach algebra is zero. By Theorem \ref{L1*}, we have the following results.
\begin{corollary}
Let $\mathcal{A}$ be a complex unital semisimple Banach $*$-algebra and $W$ be a left separating point of $\mathcal{A}$. If $\delta$ is a linear mapping from $\mathcal{A}$ into itself satisfying $\mathbb{L}_W^*$, then $\delta=0$.
\end{corollary}

Next, we prove that if $W$ is a left separating point and $WA=AW$ for all $A\in\mathcal{A}$ in the Theorem \ref{L1*}, then $\delta$ is a Jordan left derivation.

\begin{theorem}\label{L2*}
Let $\mathcal{A}$ be a unital Banach $*$-algebra, $\mathcal{M}$ be a unital $*$-$\mathcal{A}$-bimodule, and $W$ be a left separating point of $\mathcal{M}$. If $\delta$ is a linear mapping from $\mathcal{A}$ into $\mathcal{M}$ satisfying $\mathbb{L}_W^*$ and $WA=AW$ for all $A$ in $\mathcal{A}$, then $\delta$ is a Jordan left derivation.
\end{theorem}
\begin{proof}From $IW=W$, we have $\delta(W)=I\delta(W)+W{\delta(I^*)}^*$.
Since $W$ is a left separating point of $\mathcal{M}$, it follows that $\delta(I)=0$.
By Theorem \ref{L1*}, $W\delta$ is a Jordan left derivation.

Since $WA=AW$ for all $A$ in $\mathcal{A}$, we have $W\delta(A^{2})=2AW\delta(A)=2WA\delta(A)$.
It means that $\delta(A^{2})=2A\delta(A)$ for all $A$ in $\mathcal{A}$.
\end{proof}

\begin{remark}\rm
Suppose that $\mathcal{A}$ is a unital Banach $*$-algebra, $\mathcal{M}$ is a unital $*$-$\mathcal{A}$-bimodule, and $\delta$ is a linear mapping from $\mathcal{A}$ into $\mathcal{M}$ satisfying $\delta(I)=0$, the following conditions are not equivalent:
\begin{align*}
AB^*=W &\Longrightarrow \delta(W)=A{\delta(B)}^*+B^*\delta(A),\tag{L1}\label{L1} \\
A^*B=W &\Longrightarrow \delta(W)=A^*\delta(B)+B{\delta(A)}^*,\tag{L2}\label{L2}\\
AB^*=W &\Longrightarrow \delta(W)=\delta(A)B^*+{\delta(B)}^*A,\tag{R1}\label{R1}\\
A^*B=W &\Longrightarrow \delta(W)={\delta(A)}^*B+\delta(B)A^*.\tag{R2}\label{R2}
\end{align*}

If $\delta$ satisfies condition (\ref{L2}), by the similar proof of Theorem \ref{L1*}, we can prove that $TW{\delta({T^{-1}}^*)}^*+T^{-1}W{\delta(T^*)}^*=0$ for each invertible element $T$ in $\mathcal{A}$.

If $\delta$ satisfies condition (\ref{R1}),
$AB^*=W^* \Longrightarrow BA^*=W \Longrightarrow \delta(W)=\delta(B)A^*+{\delta(A)}^*B \Longrightarrow {\delta(W)}^*=A{\delta(B)}^*+B^*\delta(A)$.
Since $\delta(W)=\delta(I)W+{\delta(W^*)}^*I$, we obtain that $\delta(W^*)={\delta(W)}^*$. Then
$$AB^*=W^* \Longrightarrow \delta(W^*)=A{\delta(B)}^*+B^*\delta(A).$$
Hence $W^*\delta$ is a Jordan left derivation.
When $W$ is a right separating point of $\mathcal{M}$, we can obtain that $\delta$ is a Jordan left derivation.

If $\delta$ satisfies condition (\ref{R2}), by using above transformation, $A^*B=W^* \Longrightarrow \delta(W^*)=A^*\delta(B)+B{\delta(A)}^*$. Then $TW^*{\delta({T^{-1}}^*)}^*+T^{-1}W^*{\delta(T^*)}^*=0$ for each invertible element $T$ in $\mathcal{A}$.
\end{remark}

\begin{remark}\rm
Let $\mathcal{R}$ be a $2$-torsion free ring with the unity $I$ which satisfies that for each $T$ in $\mathcal{R}$, there is some integer $n$ such that $nI-T$ and $(n-1)I-T$ are invertible or $nI+T$ and $(n-1)I+T$ are invertible. If we replace $\mathcal{A}$ by $\mathcal{R}$ and replace linear mappings by additive mappings, then the above results are still true.
\end{remark}
\section{Property $\mathbb{T}_0^*$}

An algebra $\mathcal{A}$ is said to be \emph{zero product determined} if for every bilinear mapping $\phi$ from $\mathcal{A}\times\mathcal{A}$ into each  linear space $\mathcal{X}$
satisfying
\begin{align*}
\phi(A,B) = 0, ~~\text{whenever}~~A,B \in\mathcal{A},~~~~ AB = 0,
\end{align*}
there exists a  linear mapping $\omega$ from $\mathcal{A}$  into $\mathcal{X}$  such that $\phi(A,B) = \omega(AB)$ for all $A,B \in\mathcal{A}$.
Bre\u{s}ar \cite[Section 5.3]{BB} gave some important examples which are zero product determined algebras.

We denote by $\mathfrak{F}(\mathcal{A})$ the subalgebra of $\mathcal{A}$ generated by all idempotents in $\mathcal{A}$. Let us recall a definition that is introduced by Ghahramani in \cite{G2}.
 Let $\mathcal{A}$ be an algebra. An $\mathcal{A}$-bimodule $\mathcal{M}$ is said to have the \emph{property $\mathbb{M}$}, if there is an ideal $\mathcal{J}\subseteq\mathfrak{F}(\mathcal{A})$ of $\mathcal{A}$ such that
$$\{ M\in \mathcal{M}:XMX=0 \quad \text{for every}\quad  X\in \mathcal{J} \} = \{0\}.$$
Moreover, if $\mathcal{A}=\mathfrak{F}(\mathcal{A})$, every $\mathcal{A}$-bimodule has the property $\mathbb{M}$ . 
Note that if an $\mathcal{A}$-bimodule $\mathcal{M}$ has the property $\mathbb{M}$, then
$$\{ M\in \mathcal{M}: XM=MX=0 \quad \text{for every}\quad  X\in \mathcal{J} \} = \{0\}.$$
In the later case, $\mathcal{M}$ is said to have the \emph{property $\mathbb{P}$}.

We first consider two linear mappings satisfying $\mathbb{T}_0^*$ on  unital zero product determined $*$-algebras.

\begin{theorem}\label{0Z}
Let $\mathcal{A}$ be a unital zero product determined $*$-algebra and $\mathcal{M}$ be a unital $*$-$\mathcal{A}$-bimodule.
Suppose that $\delta $ and $\tau $ are linear mappings from $\mathcal{A}$ into $\mathcal{M}$, then $\delta $ and $\tau $ satisfy
$$A, B\in \mathcal{A},~ AB^*=0 \quad \Longrightarrow \quad \delta(A)B^*+A\tau(B)^*=0$$
if and only if there exist derivations $\Delta$ and $\Gamma$ from $\mathcal{A}$ into $\mathcal{M}$ such that $\delta(A)=\Delta(A)+\delta(I)A$, $\tau(A)=\Gamma(A)+ \tau(I) A$, and $\Delta(A^*) = \Gamma(A)^*$  for each $A$ in $\mathcal{A}$.
\end{theorem}

\begin{proof} The sufficiency is obvious, we only need to prove the necessity.
Define linear mappings $\Delta$ and $\Gamma$  from $\mathcal{A}$ into $\mathcal{M}$ by $$\Delta(A)=\delta(A)-\delta(I) A~~\text{ and }~~\Gamma(A)= \tau(A)-\tau(I) A$$ for each $A$ in $\mathcal{A}$.
It is obvious that $\Delta(I)=0=\Gamma(I)$ and
$$\Delta(A)B^*+A\Gamma(B)^*=0$$ for each $A,B \in\mathcal{A}$ with $AB^*=0$.
In the following, we show that
$\Delta$ and $\Gamma$ are derivations, and $\Delta(A)^* = \Gamma(A^*)$ for each $A$ in $\mathcal{A}$.

Define a bilinear mapping $\phi:\mathcal{A}\times \mathcal{A} \to \mathcal{M}$ by $\phi(A,B)=\Delta(A)B+A\Gamma(B^*)^*$ for each $A,B$ in $\mathcal{A}$. Thus $AB=A(B^*)^*=0$ implies $\phi(A,B)=0$. Since $\mathcal{A}$ is a zero product determined algebra, there exists a linear mapping $\omega$ from $\mathcal{A}$ into $\mathcal{M}$ such that
\begin{align}
\Delta(A)B+A\Gamma(B^*)^*=\omega(AB)\label{4.3}
\end{align}
for each $A,B$ in $\mathcal{A}$.
Now let $A=I$ in (\ref{4.3}), we obtain $\omega(B)=\Delta(I)B+\Gamma(B^*)^*$ for each $B $ in $\mathcal{A}$. It follows that $\omega(B)=\Gamma(B^*)^*$ for each $B$ in $\mathcal{A}$.
And let $B=I$ in (\ref{4.3}), $\omega(A)=\Delta(A)$ for each $A$ in $\mathcal{A}$.
Thus, $\Delta(A^*)=\omega(A^*)=\Gamma(A)^*$.
By equation (\ref{4.3}),
\begin{align*}
 \Delta(AB)&=\omega(AB)=\Delta(A)B+A\Gamma(B^*)^*\\
& = \Delta(A)B+A\omega(B)\\
&=\Delta(A)B+A\Delta(B)
\end{align*}
for each $A, B$ in $\mathcal{A}$, i.e., $\Delta$ is a derivation. One can easily verify that $\Gamma$  is also  a derivation in the same way.
\end{proof}

In the remainder of this section, we consider two linear mappings $\delta$ and $\tau$ from a unital $*$-algebra $\mathcal{A}$ into a unital $*$-$\mathcal{A}$-bimodule $\mathcal{M}$ with the property $\mathbb{P}$ satisfying ($\mathbb{T}_0^*$).

\begin{theorem}\label{0*}
Let $\mathcal{A}$ be a unital $\ast$-algebra and $\mathcal{M}$ be a unital $\ast$-$\mathcal{A}$-bimodule with the property $\mathbb{P}$. Suppose that $\delta $ and $\tau $ are linear mappings from $\mathcal{A}$ into $\mathcal{M}$, then  $\delta $ and $\tau $ satisfy
$$A,B\in \mathcal{A}, AB^*=0 \quad \Longrightarrow \quad \delta(A)B^*+A\tau(B)^*=0$$
if and only if there exist derivations $\Delta$ and $\Gamma$  from $\mathcal{A}$ into $\mathcal{M}$ such that $\delta(A)=\Delta(A)+\delta(I)A$, $\tau(A)=\Gamma(A)+\tau(I)A$ and $\Delta(A)^* = \Gamma(A^*)$ for each $A$ in $\mathcal{A}$.
\end{theorem}
\begin{proof}
We only need to prove the necessity.
Define a linear mapping $\widehat{\tau}$ from $\mathcal{A}$ into $\mathcal{M}$ by $\widehat{\tau}(A)=\tau(A^*)^*$ for each $A\in \mathcal{A}$. By assumption
$$\delta(A)B+A\tau (B^*)^*=0$$
for each $A,B$ in $\mathcal{A}$ with $AB=0$, we have
$$\delta(A)B+ A\hat{\tau}(B)=0$$
for each $A,B$ in $\mathcal{A}$ with $AB=0$. From \cite[Theorem I]{LL}, there exists a derivation $\Delta$ from $\mathcal{A}$ into $\mathcal{M}$ such that $\Delta(A)=\delta(A)-\delta(I)A=\widehat{\tau}(A)-A\widehat{\tau}(I)$ for each $A$ in $ \mathcal{A}$.
Therefore, $$\tau(A^*)-\tau(I)A^*=(\delta(A)-\delta(I)A)^*$$ for each
$A$ in $\mathcal{A}$.
Denote the linear mapping $\Gamma$ by $\Gamma(A)=\tau(A)-\tau(I)A$ for each $A$ in $\mathcal{A}$. Then $\Gamma(A^*)=\Delta(A)^* $.

It remains to show that $\Gamma$ is a derivation. Since $\Delta$ is a derivation, we have \begin{align*}\Gamma(AB)&=\Delta((AB)^*)^*=\Delta(B^*A^*)^*\\
&=(\Delta(B^*)A^*+B^*\Delta(A^*))^*\\
&=A\Delta(B^*)^*+\Delta(A^*)^*B\\
&=A\Gamma(B)+\Gamma(A)B.
\end{align*}
for each $A$, $B$ in $\mathcal{A}$. This completes the proof.
\end{proof}

In \cite{FFG}, Fadaee, Fallahi and Ghahramani considered $*$-derivable mappings at zero on unital zero product determined $\ast$-algebras or standard operator algebras. However, as a corollary of Theorem \ref{0*},
we can immediately get the following result that is an extension of \cite[Theorem 3.2]{FFG}.

\begin{corollary}
 Let $\mathcal{A}$ be a unital $\ast$-algebra and $\mathcal{M}$ be a unital $\ast$-$\mathcal{A}$-bimodule with the property $\mathbb{P}$.
Suppose that $\delta $ is a linear mapping from $\mathcal{A}$ into $\mathcal{M}$, then $\delta $ satisfies
$$A,B\in \mathcal{A}, AB^*=0 \quad \Longrightarrow \quad \delta(A)B^*+A\delta(B)^*=0$$
if and only if there exists a $*$-derivation $\Delta$ from $\mathcal{A}$ into $\mathcal{M}$ such that $\delta(A)=\Delta(A)+\delta(I)A$ for every $A$  in $\mathcal{A}$.
\end{corollary}

\section{Property $\mathbb{J}_0^*$}

In this section, we characterize two linear mappings satisfying ($\mathbb{J}_0^*$) from a unital algebra into its unital bimodule with the property $\mathbb{M}$.

We start by recalling a result concerning Jordan derivable mappings at zero Jordan products  which can be found in \cite{G2}.
\begin{lemma}\cite{G2}\label{lem1}
 Let $\mathcal{A}$ be a unital algebra and$\mathcal{M}$ be a unital $\mathcal{A}$-bimodule with the property $\mathbb{M}$. Suppose that $\delta$ is a linear mapping from $\mathcal{A}$ into $\mathcal{M}$ satisfying $$A,B\in\mathcal{A}, A\circ B=0\quad \Longrightarrow \quad  \delta(A)\circ B+A\circ\delta(B)=0.$$
Then there exists a Jordan derivation $\Delta$ from $\mathcal{A}$ into $\mathcal{M}$ such that
$$\delta(A)=\Delta(A)+A\delta(I)$$
and $A\delta(I)=\delta(I)A$ for every $A$ in $\mathcal{A}$.
\end{lemma}

The following auxiliary lemma will be needed frequently later.
\begin{lemma}\label{JL}\cite[Theorem 3.3]{G2}
If $\phi$ is a bilinear mapping from $\mathcal{A}\times\mathcal{A}$ into a vector spaces $\mathcal{X}$ such that
$$A,B \in \mathcal{A}, ~~A\circ B=0 \quad \Longrightarrow \quad \phi(A,B)=0$$
then
$$\phi(A,X)=\frac{1}{2}\phi(AX,I)+\frac{1}{2}\phi(XA,I)$$
for every $A$ in $\mathcal{A}$ and $X$ in $\mathfrak{F}(\mathcal{A})$.
\end{lemma}

\begin{lemma}\label{lem2}
 Let $\mathcal{A}$ be a unital algebra and $\mathcal{M}$ be a unital $\mathcal{A}$-bimodule with the property $\mathbb{M}$.
 Suppose that $\delta$ is a linear mapping from $\mathcal{A}$ into $\mathcal{M}$ satisfying $$A,B\in\mathcal{A}, A\circ B=0\quad \Longrightarrow \quad  \delta(A)\circ B-A\circ\delta(B)=0.$$
Then $\delta(A)=\frac{1}{2}(A\delta(I)+\delta(I)A)$ for every $A $ in $\mathcal{A}$.
\end{lemma}


\begin{proof}
Define a bilinear mapping $\phi$ from $\mathcal{A}\times \mathcal{A}$ into $\mathcal{M}$ by $\phi(A, B) = \delta(A)\circ B-A\circ\delta(B)$ for every $A, B $ in $\mathcal{A}$. By the definition of $\phi$, it follows that
 $A\circ B = 0$ implies $\phi(A, B)=0$. In account of  Lemma \ref{JL}, for every $A$ in $\mathcal{A}$ and every $X$ in $\mathcal {J}\subset \mathfrak{F}(\mathcal{A})$, we have
\begin{align}
\delta(I)\circ X-I \circ \delta(X)&=\delta(X)\circ I- X \circ \delta(I), \label{5.1} \\
\delta(A)\circ X-A \circ\delta(X)&=\frac{1}{2}[\delta(AX)\circ I-AX \circ\delta(I)+ \delta(XA)\circ I - XA \circ \delta(I)]\nonumber\\
&=\delta(A X+XA)-\frac{1}{2}[(A X+XA)\circ \delta(I)].\label{5.2}
\end{align}
It follows  from (\ref{5.1}) that
\begin{align}\label{5.3}
 X \circ\delta(I)=2\delta(X).
 \end{align}
Since $\mathcal {J}$ is an ideal,
$AX \circ\delta(I)=2\delta(AX)$, $ XA \circ\delta(I)=2\delta(XA)$ for all $A$ in $\mathcal{A}$ and $X$ in $\mathcal {J}$.
Then $\delta(AX+XA)=\frac{1}{2}(A X+XA)\circ\delta(I)$. In view of (\ref{5.2}),
\begin{align}\label{5.4}
A \circ\delta(X)=\delta(A)\circ X
\end{align}
for every $A$ in $\mathcal{A}$ and $X$ in $\mathcal{J}$.

Define a linear mapping $\Delta: \mathcal{A} \to \mathcal{M}$ by $\Delta(A)=\delta(A)-\frac{1}{2}(A\delta(I)+\delta(I)A)$ for every $A $ in $\mathcal{A}$.
It is easy to verify  $\Delta(X)=0= \Delta(AX)=\Delta(XA) $. It is enough to prove that $\Delta=0$.
Using (\ref{5.3}) together with (\ref{5.4}), we arrive at
\begin{align}\label{5.5}
\Delta(A)\circ X&=\delta(A)\circ X-  \frac{1}{2}[\delta(I)AX+A\delta(I)X+X\delta(I)A+XA\delta(I)]\nonumber\\
&= A\circ\delta(X)-\frac{1}{2}[\delta(I)AX+A\delta(I)X+X\delta(I)A+XA\delta(I)]\nonumber\\
&=A\circ\Delta(X) +\frac{1}{2}[\delta(I)XA+X\delta(I)A+A\delta(I)X+AX\delta(I)]\nonumber\\
&\quad  - \frac{1}{2}[\delta(I)AX+A\delta(I)X+X\delta(I)A+XA\delta(I)]\nonumber\\
&=\frac{1}{2}[\delta(I)(XA-AX)+(AX-XA)\delta(I)].
\end{align}
Replacing $A$ by $AX$ and $XA$ respectively in (\ref{5.5}), one can obtain
\begin{align}
0=\Delta(AX)\circ X&=\frac{1}{2}[\delta(I)(XAX-AX^2)+(AX^2-XAX)\delta(I)]\label{5.6}\\
0=\Delta(XA)\circ X&=\frac{1}{2}[\delta(I)(X^2A-XAX)+(XAX-X^2A)\delta(I)]\label{5.7}.
\end{align}
Adding (\ref{5.6}) and (\ref{5.7}), we get
\begin{align}\label{5.8}
0= \delta(I)(X^2A-AX^2) + (AX^2-X^2A ) \delta(I).
\end{align}
In particular, for any idempotent $P$, by (\ref{5.5}) and (\ref{5.8}), we conclude that
\begin{align}
 0& =\delta(I)(PA-AP)+(AP-PA)\delta(I)\nonumber\\
 &=2\Delta(A)\circ P\nonumber\\
 &=2\Delta(A)P+2P\Delta(A).\label{5.9}
\end{align}
Multiplying $P$ on both sides, the left and the right of (\ref{5.9}) respectively,  we arrive at  $0=P\Delta(A)P=P\Delta(A)=\Delta(A)P$ for any idempotent $P$. Hence $0=X\Delta(A)X=\Delta(A)X=X\Delta(A)$ for all $X\in \mathcal{J}$. It follows from property $\mathbb{M}$ that $\Delta(A)=0$ for all $A\in \mathcal{A}$.
\end{proof}
\begin{theorem}\label{0JM}
 Let $\mathcal{A}$ be a unital algebra and $\mathcal{M}$ be a unital $\mathcal{A}$-bimodule with the property $\mathbb{M}$.
 Suppose that $\delta$ and $\tau$ are linear mappings from $\mathcal{A}$ into $\mathcal{M}$ satisfying $$A,B\in\mathcal{A}, A\circ B=0\quad \Longrightarrow \quad  \delta(A)\circ B+A\circ\tau(B)=0.$$
Then there exists a Jordan derivation $\Delta$ from $\mathcal{A}$ into $\mathcal{M}$ such that
$$\delta(A)=\Delta(A)+\delta(I)A, ~~\tau(A)=\Delta(A)+ A\tau(I)$$
for every $A$ in $\mathcal{A}$.
\end{theorem}
\begin{proof}
For all $A,B\in\mathcal{A}$ with $A\circ B=0$,  we have $\delta(A)\circ B+A\circ\tau(B)=0$  and $\delta(B)\circ A+B\circ\tau(A)=0$. Comparing the above equations, we get
\begin{align*}
(\delta+\tau)(A)\circ B+A\circ(\delta+\tau)(B)=0,\\
(\delta-\tau)(A)\circ B-A\circ(\delta-\tau)(B)=0.
\end{align*}
It follows from Lemma \ref{lem1} that there exists a Jordan derivation $\Delta'$ such that $(\delta+\tau)(A)=\Delta'(A)+(\delta+\tau)(I)A$, where $(\delta+\tau)(I)A=A(\delta+\tau)(I)$ for all $A\in \mathcal{A}$.
According to Lemma \ref{lem2}, we obtain $(\delta-\tau)(A)=\frac{1}{2} [A(\delta-\tau)(I)+(\delta-\tau)(I)A]$ for all $A\in \mathcal{A}$.
We conclude that
\begin{align*}
2\delta(A)&=\Delta'(A)+(\delta+\tau)(I) A +\frac{1}{2}[A(\delta-\tau)(I)-(\delta-\tau)(I)A]+(\delta-\tau)(I)A\\
&=\Delta'(A)+\frac{1}{2}[A(\delta-\tau)(I)-(\delta-\tau)(I)A] +2\delta(I)A,\\
2\tau(A)&=\Delta'(A)+A(\delta+\tau)(I)+
\frac{1}{2}[A(\delta-\tau)(I)-(\delta-\tau)(I)A]-A(\delta-\tau)(I)\\
&=\Delta'(A)+\frac{1}{2}[A(\delta-\tau)(I)-(\delta-\tau)(I)A]+2A\tau(I).
\end{align*}
 for all $A$ in $\mathcal{A}$. It is clear that $\Delta:=\frac{1}{2}\Delta' +\frac{1}{4}[A(\delta-\tau)(I)-(\delta-\tau)(I)A]$ is also a Jordan derivation.
\end{proof}

In particular, by Theorem \ref{0JM}, we obtain the following corollary  which generalized  Bahmani and Ghomanjani's result \cite{cen} on zero Jordan product determined algebras.
\begin{corollary}
 Let $\mathcal{A}$ be a unital algebra and $\mathcal{M}$ be a unital $\mathcal{A}$-bimodule with the property $\mathbb{M}$.
 Suppose that $\varphi$ is a linear mapping from $\mathcal{A}$ into $\mathcal{M}$ satisfying
 $$A,B\in\mathcal{A}, A\circ B=0\quad \Longrightarrow \quad  A \circ \varphi(B) =0.$$
Then $\varphi(A)=\varphi(I)A=A\varphi(I)$ for every $A$ in $\mathcal{A}$.
\end{corollary}
\begin{proof}
First, taking $\delta=0$ and $\tau=\varphi$ in Theorem \ref{0JM}, we have $\varphi(A)=A\varphi(I)$.
On the other hand, we also have $B \circ \varphi(A)=0$ for all $A, B\in \mathcal{A}$ with $B\circ A=0$. Taking $\delta=\varphi$, $\tau=0$ in Theorem \ref{0JM}, we obtain
$\varphi(A)=\varphi(I)A$.
\end{proof}

By  Theorem \ref{0JM} and the proof of Theorem \ref{0*},
we now state and prove our result in this section.

\begin{corollary} \label{0JM*}
Let $\mathcal{A}$ be a unital $\ast$-algebra and $\mathcal{M}$  be a unital $\ast$-$\mathcal{A}$-bimodule with the property $\mathbb{M}$.
If $\delta $ and $\tau $ are linear mappings from $\mathcal{A}$ into $\mathcal{M}$ satisfying
$$A,B\in \mathcal{A}, A\circ B^*=0 \quad \Longrightarrow \quad \delta(A)\circ B^*+A\circ\tau(B)^*=0,$$
then there exist Jordan derivations $\Delta$ and $\Gamma$ from $\mathcal{A}$ into $\mathcal{M}$ such that $\delta(A)=\Delta(A)+\delta(I)A$, $\tau(A)=\Gamma(A)+ \tau(I)A$, and $\Delta(A^*) = \Gamma(A)^*$  for every $A$ in $\mathcal{A}$.
\end{corollary}
\begin{proof}
Define a linear mapping $\widehat{\tau}$ from $\mathcal{A}$ into $\mathcal{M}$ by $\widehat{\tau}(A)=\tau(A^*)^*$, for every $A$ in $\mathcal{A}$. Then
$$A\circ \widehat{\tau}(B)+\delta(A)\circ B=A\circ\tau (B^*)^*+\delta(A)\circ B=0$$
for each $A,B$ in $\mathcal{A}$ with $A\circ B=A\circ B^{**}=0$.
Thus $\delta$ and $\widehat{\tau}$ satisfy the conditions of Theorem \ref{0JM},
there exists a Jordan derivation $\Delta$ from $\mathcal{A}$ into $\mathcal{M}$ such that $$\Delta(A)=\delta(A)-\delta(I)A
=\widehat{\tau}(A)-A\widehat{\tau}(I) .$$
Denote the linear mapping $\Gamma$ from $\mathcal{A}$ into $\mathcal{M}$ by $\Gamma(A)=\tau(A)- \tau(I)A$ for every $A$ in $\mathcal{A}$.
Therefore
$$\Delta(A)=\delta(A)-A\delta(I)
=(\tau(A^*)- \tau(I)A^*)^*=\Gamma(A^*)^*.$$
At the same time,
$$\Gamma(A^2)=\Delta((A^2)^*)^*=\Delta((A^*)^2)^*
=[\Delta(A^*)A^*+A^*\Delta(A^*)]^*=A\Gamma(A)+\Gamma(A)A$$
for every $A$ in $\mathcal{A}$, i.e., $\Gamma$ is a Jordan derivation.
\end{proof}

\begin{remark}
If $\mathcal{A}$ is a unital $\ast$-algebra and $\mathcal{M}$ is a unital $\ast$-$\mathcal{A}$-bimodule with the property $\mathbb{M}$, and $\delta $ is a linear mapping from $\mathcal{A}$ into $\mathcal{M}$. The authors \cite[Theorem 3.6]{AHL2} showed that if $\delta $ satisfies
$$A,B\in \mathcal{A}, A\circ B^*=0 \quad \Longrightarrow \quad \delta(A)\circ B^*+A\circ\delta(B)^*=0,~~\delta(I)A=A\delta(I),$$ then
there exists a Jordan derivation $\Delta$ from $\mathcal{A}$ into $\mathcal{M}$ such that $\delta(A)=\Delta(A)+\delta(I)A$ and $\Delta(A^*) = \Delta(A)^*$, for every $A$ in $\mathcal{A}$. In fact, by Theorem \ref{0JM*}, the above condition $\delta(I)A=A\delta(I)$ is unnecessary if $\delta=\tau$.
\end{remark}

\end{document}